\documentclass[11pt,a4paper]{amsart}

\usepackage[margin=1.25in,a4paper]{geometry}
\usepackage{lmodern}
\usepackage[utf8]{inputenc}
\usepackage[T1]{fontenc}

\usepackage{microtype}
\usepackage[shortlabels]{enumitem}
\usepackage{svg}

\usepackage[english]{babel}
\usepackage{mathrsfs}
\usepackage{amssymb}

\usepackage{xcolor}
\definecolor{linkcol}{RGB}{0,80,158}
\definecolor{citecol}{RGB}{46,117,120}

\usepackage[colorlinks]{hyperref}
\hypersetup{
    colorlinks=true,
    linkcolor=linkcol,
    urlcolor=linkcol,
    citecolor=citecol,
}
\usepackage[hyperpageref]{backref}
\usepackage[msc-links,alphabetic]{amsrefs}
\usepackage[capitalize,nameinlink,noabbrev]{cleveref}
\crefname{equation}{}{}
\numberwithin{equation}{section}

\theoremstyle{plain}
\newtheorem{theorem}{Theorem}[section]
\newtheorem{lemma}[theorem]{Lemma}
\newtheorem{corollary}[theorem]{Corollary}

\theoremstyle{definition}

\theoremstyle{remark}

\newcommand{\bD}{\mathbb{D}}
\newcommand{\bQ}{\mathbb{Q}}

\newcommand{\bT}{\mathbb{T}}
\newcommand{\bZ}{\mathbb{Z}}

\newcommand{\fA}{\mathfrak{A}}
\newcommand{\abs}[1]{\lvert#1\rvert}

\newcommand{\norm}[1]{\lVert#1\rVert}

\newcommand{\dd}{\mathrm{d}}

\title{Exotic operators on the Hardy spaces of the infinite-dimensional torus}

\author{Viktor Andersson}
\address{Department of Mathematical Sciences, Norwegian University of Science and Technology (NTNU), 7491 Trondheim, Norway}
\email{viktor.andersson@ntnu.no}
\date{\today}

\begin{document}
\begin{abstract}
We quantify the failure of interpolation between Hardy spaces $H^p(\bT^\infty)$ on the infinite-dimensional torus. For any set $A\subseteq[1,\infty]$ such that $A\setminus\{\infty\}$ is closed in $[1,\infty)$, we show that there exists an operator that is densely defined on $H^p(\bT^\infty)$ for all $1\leq p<\infty$ and weak-$\ast$ densely defined on $H^\infty(\bT^\infty)$ that extends to a bounded linear operator on $H^p(\bT^\infty)$ if and only if $p\in A$.
\end{abstract}
\thanks{The author is supported by Grant 354537 of the Research Council of Norway.}
\maketitle

\section{Introduction}\label{section:introduction}

Let $L^p(\bT^\infty)$ be the usual $L^p$-space of the infinite-dimensional torus $\bT^\infty=\bT\times\bT\times\cdots$ equipped with its unique normalized Haar measure $m_\infty$, which is the countable product measure of the normalized arclength measure $m$ on each circle $\bT$. The dual group of $\bT^\infty$ can be identified with the positive rational numbers $\bQ^+$ with the discrete topology by identifying the rational number $q\in\bQ^+$ with the map
$$\chi\mapsto\chi(q)=\prod_j\chi_j^{k_j}$$
with $\chi=(\chi_1,\chi_2,\dots)\in\bT^\infty$ and $k_j$ the unique integer exponent of the $j$'th prime number in the prime factorization of $q$. The Hardy spaces of $\bT^\infty$ are then defined as
$$H^p(\bT^\infty)=\{f\in L^p(\bT^\infty):\widehat f(q)=0\text{ for all }q\in\bQ^+\setminus\bZ^+\},$$
where
$$\widehat f(q)=\int_{\bT^\infty}f(\chi)\overline{\chi(q)}\,\dd m_\infty(\chi)$$
are the Fourier coefficients of $f$ for $q\in\bQ^+$. In the finite-dimensional setting, the standard approach to interpolation between the Hardy spaces $H^p(\bT^d)$ goes through the boundedness of the Riesz projection $L^p(\bT^d)\to H^p(\bT^d)$. In the infinite-dimensional setting, the corresponding Riesz projection is unbounded for all $p\neq2$ (the norm of the Riesz projection on $\bT^d$ is unbounded in $d$, see \cite{hollenbeck_best_2000}), and so one no longer has access to the standard method of interpolation in this setting. One can view this as a failure in duality: if $p\neq2$, then the space $H^p(\bT^\infty)$ is not complemented in $L^p(\bT^\infty)$ (see \cite{bayart_interpolation_2019}*{Corollary 3.7}), and furthermore its dual space $H^p(\bT^\infty)^*$ does not correspond to $H^q(\bT^\infty)$ with $q$ the Hölder conjugate of $p$ (see \cite{saksman_integral_2009}*{Section 3}). This leads naturally to the question of whether interpolation simply fails between the spaces $H^p(\bT^\infty)$, or whether this is simply a failure of existing methods. Our main result tells us that interpolation fails spectacularly; we construct operators with exotic infinite-dimensional behavior---operators for which we can prescribe the $p$ for which one has boundedness with very large freedom.

\begin{theorem}\label{theorem:main-theorem}
    Let $A\subseteq[1,\infty]$ be such that $A\setminus\{\infty\}$ is closed in $[1,\infty)$. Then there exists a linear operator $T$, densely defined on $H^p(\bT^\infty)$ for all $1\leq p<\infty$ and weak-$\ast$ densely defined on $H^\infty(\bT^\infty)$, that extends to a bounded operator on $H^p(\bT^\infty)$ if and only if $p\in A$.
\end{theorem}

The requirement that $A\setminus\{\infty\}$ be closed in \cref{theorem:main-theorem} is necessary for our construction to work: our method is based on writing $[1,\infty)\setminus A$ as a countable union of relatively open subintervals of $[1,\infty)$. We are unaware of whether this completely characterizes all possible sets $A$, and find it an interesting question to determine precisely what sets $A$ can occur.

\cref{theorem:main-theorem} provides a significant generalization of a result of Brevig, Ortega-Cerdà and Seip, who proved the result in the special case of $A=\{2,4,\dots,2(n+1)\}$ for a fixed positive integer $n$ \cite{brevig_idempotent_2021}*{Theorem 1.4}. Motivated by the local embedding problem for the Hardy spaces of Dirichlet series (see \cite{brevig_idempotent_2021}*{Section 4} for details), they asked also whether the case when $A=2\bZ^+\cup\{\infty\}$ can occur \cite{brevig_idempotent_2021}*{Problem 4.1}. As this set satisfies the assumption of \cref{theorem:main-theorem}, we obtain a positive answer to their question. Taking $A=\{p_1,p_2\}$ for any $1\leq p_1<p_2\leq\infty$, we obtain also the following corollary, which improves on \cite{bayart_interpolation_2019}*{Theorems 3.3 and 3.5} of Bayart and Masty\l o.

\begin{corollary}\label{corollary:not-interpolation-space}
    $H^p(\bT^\infty)$ is not an interpolation space for the pair $(H^{p_1}(\bT^\infty),H^{p_2}(\bT^\infty))$ for any $1\leq p_1<p<p_2\leq\infty$.
\end{corollary}

This result is new for any $1\leq p_1<p_2\leq\infty$ with the exception of when $p_1$ and $p_2$ are two consecutive even integers as a consequence of \cite{brevig_idempotent_2021}*{Theorem 1.4}.

We obtain also a quantitative consequence for interpolation between Hardy spaces of finite-dimensional tori: the best interpolation constant for $H^p(\bT^d)$ grows exponentially in $d$ (see \cref{theorem:asymptotic-of-finite-interpolation-constant} below).

Our construction is done in three main steps. We start by fixing an arbitrary relatively open and bounded subinterval $I$ of $[1,\infty)$ with endpoints $p_1<p_2$. For any $p'\in I$, we then construct an operator $T$ with the property that its norm on $H^{p'}(\bT^\infty)$ is strictly greater than the supremum of its norm on $H^p(\bT^\infty)$ for all $p\in[1,\infty]\setminus I$. This operator will be of the form
$$Tf(\chi)=\frac{1}{2^n}\langle f,\Phi_n\rangle(\chi_1+\chi_2)^m$$
where $\langle\,\cdot\,,\,\cdot\,\rangle$ denotes the usual integral pairing on $\bT^\infty$, and
$$\Phi_n(\chi)=(\chi_1+\chi_2)(\chi_3+\chi_4)\cdots(\chi_{2n-1}+\chi_{2n}).$$
In particular, the key part of the argument comes from the fact that we can explicitly compute the norm of $T$ on $H^p(\bT^\infty)$ as
$$\norm{T}_{p\to p}=\frac{\norm{\chi_1+\chi_2}_{mp}^m}{\norm{\chi_1+\chi_2}_p^n}$$
as a consequence of \cref{lemma:dual-norm} below. This quantity can be further expressed in terms of the beta function, and so by a special function argument we show that we can choose integers $n$ and $m$ so that the norm of $T$ satisfies the desired inequalities (compare with \cref{fig:plot-of-norm} below).

The second step of the proof consists of a standard infinite tensorization argument by exploiting the fact that functions on $\bT^\infty$ have infinitely many variables. This allows us to scale the norm in a way where it blows up exponentially for $p$ in a neighborhood of $p'$ but stays bounded for $p\in[1,\infty]\setminus I$. By considering an appropriate weighted infinite sum, we then obtain an operator that is bounded for $p\in[1,\infty]\setminus I$ but unbounded for $p$ in a neighborhood of $p'$.

For the last step, we use the second step together with a compactness argument to construct an operator that is bounded on $H^p(\bT^\infty)$ for $p\in A\cup\{\infty\}$ and unbounded for all other $p$. To deal with the endpoint case where $\infty\notin A$, we use a result of Anderson, Jovovic and Smith \cite{anderson_integral_2014} to find a Volterra operator on the Hardy spaces of the disk $\bD$ that maps the polynomials into $H^\infty(\bD)$, that is bounded on $H^p(\bD)$ for $1\leq p<\infty$, and that is unbounded on $H^\infty(\bD)$. By transferring this operator to the infinite polytorus in the natural way and adding it to our previously constructed operator, we obtain our result.

\begin{figure}[h]
    \centering
    \includegraphics[width=0.6\linewidth]{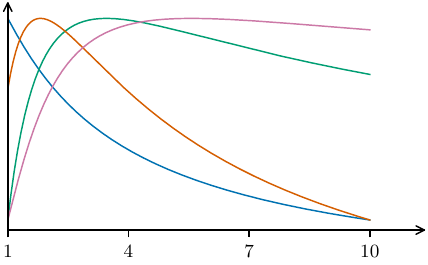}
    \caption{Plots of $p\mapsto\norm{\chi_1+\chi_2}_{mp}^m/\norm{\chi_1+\chi_2}_p^n$ for \textcolor[HTML]{0072B2}{$(n,m)=(3,2)$}, \textcolor[HTML]{009E73}{$(n,m)=(3,9)$}, \textcolor[HTML]{D55E00}{$(n,m)=(5,15)$}, and \textcolor[HTML]{CC79A7}{$(n,m)=(6,1000)$}, scaled to all take values in the same range.}
    \label{fig:plot-of-norm}
\end{figure}

\subsection*{Organization.}
The remainder of the paper is split into two sections. The first of these, \cref{section:proof-of-main-theorem}, is dedicated to the proof of \cref{theorem:main-theorem}, and in the second one, \cref{section:finite-dimensional-consequences}, we apply our methods to derive asymptotics for the best interpolation constants between Hardy spaces on finite-dimensional tori.

\section{Proof of \texorpdfstring{\cref{theorem:main-theorem}}{Theorem \ref{theorem:main-theorem}}}\label{section:proof-of-main-theorem}

Our argument will be based on studying the norm of the polynomial
$$\varphi(\chi)=\chi_1+\chi_2$$
in two variables. In particular, a key role will be played by the well-known fact that its norm can be explicitly computed in terms of the beta function
$$B(x,y)=\int_0^1 t^{x-1}(1-t)^{y-1}\,\dd t.$$

\begin{lemma}\label{lemma:norm-of-phi}
    The function $\varphi(\chi)=\chi_1+\chi_2$ has norm
    $$\norm{\varphi}_p=2\left(\frac{1}{\pi}B\left(\frac{1}{2},\frac{p+1}{2}\right)\right)^{1/p}$$
    for $1\leq p<\infty$ and
    $$\norm{\varphi}_\infty=2.$$
\end{lemma}

\begin{proof}
    The case when $p=\infty$ is clear, so take $1\leq p<\infty$. Using Fubini's theorem, the rotational invariance of $m$ and symmetry, we compute
    $$\norm{\varphi}_p^p=\int_\bT\int_\bT\abs{\chi_1+\chi_2}^p\,\dd m(\chi_1)\,\dd m(\chi_2)=\frac{1}{2\pi}\int_0^{2\pi}\abs{1+e^{i\theta}}^p\,\dd\theta=\frac{2^p}{\pi}\int_0^\pi\cos\left(\frac{\theta}{2}\right)^p\,\dd\theta.$$
    From here the result follows after a change of variables.
\end{proof}

At several points in our arguments, we will use that we can identify $H^p(\bT^d)$ with a subspace of $H^p(\bT^\infty)$, namely the subspace
$$H^p(\bT^d)=\{f\in H^p(\bT^\infty):\widehat f(n)=0\text{ if }p_j\mid n\text{ for some }j\geq d+1\}.$$
That is, $H^p(\bT^d)$ is the subspace of $H^p(\bT^\infty)$ consisting of those functions that only depend on the first $d$ variables. The projections $\fA_d:H^p(\bT^\infty)\to H^p(\bT^d)$ obtained by formally setting $\widehat{\fA_d f}(n)=0$ if $p_j\mid n$ for some $j\geq d+1$ and $\widehat{\fA_d f}(n)=\widehat f(n)$ otherwise are called \emph{die Abschnitte}, and define contractive projections onto $H^p(\bT^d)$, which is easily seen by writing
$$\fA_df(\chi)=\int_{\bT^\infty}f(\chi_1,\dots,\chi_d,\eta_1,\eta_2,\dots)\,\dd m_\infty(\eta).$$
Die Abschnitte will be relevant to our construction, as they allow us to isolate operators that only depend on the first $d$ variables, or more precisely, on the Fourier coefficients supported on the integers generated by the first $d$ primes.

To construct our operator, we will next consider the polynomial $\Phi_n$ defined in the introduction. This polynomial defines a bounded linear functional on $H^p(\bT^\infty)$ for each $1\leq p\leq\infty$ by the formula $f\mapsto\langle f,\Phi_n\rangle$, and we will write $\norm{\Phi_n}_{H^p(\bT^\infty)^*}$ to denote the norm of this functional. The main property we need of this functional is that we can compute its norm explicitly.

\begin{lemma}\label{lemma:dual-norm}
    For any positive integer $n$ and any $1\leq p\leq\infty$, it holds that
    $$\norm{\Phi_n}_{H^p(\bT^\infty)^*}=\frac{2^n}{\norm{\varphi}_p^n}.$$
\end{lemma}

\begin{proof}
    For a positive integer $j$, let $\varphi_j(\chi)=\chi_{2j-1}+\chi_{2j}$. By \cite{brevig_linear_2019}*{Lemma 5} it is the case that
    $$\norm{\varphi_j}_{H^p(\bT^\infty)^*}=\frac{2}{\norm{\varphi}_p}.$$
    Observing that the functional $f\mapsto\langle f,\Phi_n\rangle$ is the tensor product of the functionals $f\mapsto\langle f,\varphi_j\rangle$ acting on the variables $\chi_{2j-1},\chi_{2j}$ for $j=1,\dots,n$ composed with $\fA_{2n}$, we obtain that
    \[\norm{\Phi_n}_{H^p(\bT^\infty)^*}=\norm{\varphi_1}_{H^p(\bT^\infty)^*}\cdots\norm{\varphi_n}_{H^p(\bT^\infty)^*}=\frac{2^n}{\norm{\varphi}_p^n}.\qedhere\]
\end{proof}

With this, we can now define the family of operators that will be central to our construction. For positive integers $n$ and $m$, consider the operator $T_{n,m}$ defined by
$$T_{n,m}f=\frac{1}{2^n}\langle f,\Phi_n\rangle\varphi^m.$$
Then $T_{n,m}$ is a bounded operator on $H^p(\bT^\infty)$ for all $1\leq p\leq\infty$, and we can compute its norm as
$$\norm{T_{n,m}}_{p\to p}=\frac{1}{2^n}\norm{\Phi_n}_{H^p(\bT^\infty)^*}\norm{\varphi^m}_p=\frac{\norm{\varphi}_{mp}^m}{\norm{\varphi}_p^n}.$$
Consider now the function $G$ defined by
$$G(p)=\frac{1}{p}\log\left(\frac{1}{\pi}B\left(\frac{1}{2},\frac{p+1}{2}\right)\right)$$
for $1\leq p<\infty$. By \cref{lemma:norm-of-phi} it is then the case that
$$\log\norm{T_{n,m}}_{p\to p}+(n-m)\log 2=mG(mp)-nG(p)$$
for $1\leq p<\infty$. We shall therefore be concerned with the properties of $G$. We start with the observation that $G$ is a strictly increasing function.

\begin{lemma}\label{lemma:G-strictly-increasing}
    The function $G$ is strictly increasing.
\end{lemma}

\begin{proof}
    Differentiating $G$ by using the identity
    $$\partial_y B(x,y)=B(x,y)(\psi(y)-\psi(x+y))$$
    where $\psi=(\log\Gamma)'$ is the digamma function, we can compute
    $$G'(p)=\frac{1}{2p}\left(\psi\left(\frac{p}{2}+\frac{1}{2}\right)-\psi\left(\frac{p}{2}+1\right)\right)-\frac{1}{p^2}\log\left(\frac{1}{\pi}B\left(\frac{1}{2},\frac{p+1}{2}\right)\right).$$
    Set $H(p)=p^2G'(p)$ so that $G'(p)>0$ if and only if $H(p)>0$. Observe first that
    $$H(1)=\frac{1}{2}\left(\psi\left(1\right)-\psi\left(\frac{3}{2}\right)\right)-\log\left(\frac{1}{\pi}B\left(\frac{1}{2},1\right)\right)=\log\pi-1>0.$$
    Next, one computes
    $$H'(p)=\frac{p}{4}\left(\psi_1\left(\frac{p}{2}+\frac{1}{2}\right)-\psi_1\left(\frac{p}{2}+1\right)\right)$$
    where $\psi_1=\psi'$ is the trigamma function. As the trigamma function is strictly decreasing on $(0,\infty)$, it follows that $H'(p)>0$ for all $p\in[1,\infty)$. Thus $H$ is increasing, so as $H(1)>0$, it follows that $H(p)=p^2G'(p)>0$ for all $p\in[1,\infty)$.
\end{proof}

We will also need the following asymptotic expansions of $G$ and its first three derivatives.

\begin{lemma}\label{lemma:asymptotic-of-G}
    It holds that
    $$G(p)=\frac{1}{2p}\log\frac{2}{\pi p}+O\left(\frac{1}{p^2}\right),\quad G'(p)=-\frac{1}{2p^2}\log\frac{2}{\pi p}-\frac{1}{2p^2}+O\left(\frac{1}{p^3}\right),$$
    $$G''(p)=\frac{1}{p^3}\log\frac{2}{\pi p}+\frac{3}{2p^3}+O\left(\frac{1}{p^4}\right),\quad G'''(p)=-\frac{3}{p^4}\log\frac{2}{\pi p}-\frac{11}{2p^4}+O\left(\frac{1}{p^5}\right)$$
    as $p\to\infty$.
\end{lemma}

\begin{proof}
    The result readily follows by differentiating $G$ three times, using the identity $B(x,y)=\Gamma(x)\Gamma(y)/\Gamma(x+y)$ together with standard asymptotic expansions of $\log\Gamma$ and its first three derivatives (e.g., \cite{DLMF}*{\href{https://dlmf.nist.gov/5.11.E1}{(5.11.1)}, \href{https://dlmf.nist.gov/5.11.E2}{(5.11.2)}, \href{https://dlmf.nist.gov/5.15.E9}{(5.15.9)}}), and simplifying.
\end{proof}

The main part of showing that we can choose $m$ and $n$ with the property we want comes from the following lemma about the function $G$.

\begin{lemma}\label{lemma:G-fraction-inequality-not-at-one}
    Let $1\leq p_1<p<p_2<\infty$. Then there exists a $t_0>0$ such that
    $$\frac{G(tp)-G(tp_1)}{G(p)-G(p_1)}-\frac{G(tp_2)-G(tp)}{G(p_2)-G(p)}>\frac{1}{t}$$
    for all $t\geq t_0$.
\end{lemma}

\begin{proof}
    Using \cref{lemma:asymptotic-of-G} we have that
    \begin{multline*}
    \frac{G(tp)-G(tp_1)}{G(p)-G(p_1)}-\frac{G(tp_2)-G(tp)}{G(p_2)-G(p)} \\
        =\frac{\log t}{2t}\left(\frac{\frac{1}{p_1}-\frac{1}{p}}{G(p)-G(p_1)}-\frac{\frac{1}{p}-\frac{1}{p_2}}{G(p_2)-G(p)}\right)+O\left(\frac{1}{t}\right)\quad\text{as}\quad t\to\infty.
    \end{multline*}
    We see from this that it suffices to show that
    $$\frac{G(p_2)-G(p)}{\frac{1}{p}-\frac{1}{p_2}}>\frac{G(p)-G(p_1)}{\frac{1}{p_1}-\frac{1}{p}},$$
    and for this it suffices to show that the function $R$ defined by $R(r)=G(1/r)$ for $r\in(0,1]$ is strictly convex. Differentiating $R$ using the same identities as in the proof of \cref{lemma:G-strictly-increasing}, we can compute
    $$R''(r)=\frac{1}{4r^3}\left(\psi_1\left(\frac{1}{2r}+\frac{1}{2}\right)-\psi_1\left(\frac{1}{2r}+1\right)\right),$$
    and so the strict convexity of $R$ follows from the fact that the trigamma function is strictly decreasing on $(0,\infty)$.
\end{proof}

The final thing we need concerns the uniqueness of the maximal value of the norm $\norm{T_{n,m}}_{p\to p}$ as a function of $p$ for sufficiently large $m$. For convenience, we will write
$$F_{n,m}(p)=\norm{T_{n,m}}_{p\to p}$$
for this function. It should be noted that $F_{n,m}$ is continuous on $[1,\infty]$ and smooth on $(1,\infty)$ as a consequence of \cref{lemma:norm-of-phi}.

\begin{lemma}\label{lemma:existence-of-m}
    There exists a positive integer $m_0$ such that the following hold for all positive integers $n$ and $m$ with $m\geq m_0$:
    \begin{enumerate}[(i)]
        \item If $F'_{n,m}(p')>0$ for some $1<p'<\infty$, then $F'_{n,m}(p)>0$ for all $p\in(1,p']$.
        \item If $F'_{n,m}(p')<0$ for some $1<p'<\infty$, then $F'_{n,m}(p)<0$ for all $p\in[p',\infty)$.
        \item If $F'_{n,m}(p')=0$ for some $1<p'<\infty$, then $F'_{n,m}(p)>0$ for all $p\in(1,p')$, and $F'_{n,m}(p)<0$ for all $p\in(p',\infty)$.
    \end{enumerate}
\end{lemma}

\begin{proof}
    As $F_{n,m}$ is always positive, we have that $F'_{n,m}$ and $(\log F_{n,m})'$ have the same sign. Compute
    $$(\log F_{n,m})'(p)=m^2G'(mp)-nG'(p)$$
    and set
    $$h_m(p)=\frac{G'(mp)}{G'(p)}$$
    and note that, as $G'$ is a positive function by \cref{lemma:G-strictly-increasing}, we have that $F_{n,m}'(p)=0$ if and only if $h_m(p)=n/m^2$, that $F_{n,m}'(p)>0$ if and only if $h_m(p)>n/m^2$, and that $F_{n,m}'(p)<0$ if and only if $h_m(p)<n/m^2$. We thus see that it suffices to show that there exists some positive integer $m_0$ such that $h_m'(p)<0$ for all integers $m\geq m_0$ and all $1<p<\infty$. By differentiating $h_m$ one sees that $h_m'(p)<0$ if and only if
    \begin{equation}\label{eq:G-mp-p-inequality}
    \frac{mpG''(mp)}{G'(mp)}<\frac{pG''(p)}{G'(p)}.
    \end{equation}
    Set $g(p)=pG''(p)/G'(p)$. Differentiating $g$ using \cref{lemma:asymptotic-of-G} and simplifying, we can expand
    $$g'(p)=\frac{1}{G'(p)^2}\left(-\frac{1}{4p^5}+O\left(\frac{\log p}{p^6}\right)\right)\quad\text{as}\quad p\to\infty.$$
    In particular, we can find a $p_0\geq 1$ such that $g'(p)<0$ for all $p\geq p_0$. We claim next that $g(p)>-2$ for all $p\geq1$. Consider $H(p)=p^2G'(p)$. It was shown in the proof of \cref{lemma:G-strictly-increasing} that $H'(p)>0$ for all $p\geq1$. In particular, as
    $$(g(p)+2)pG'(p)=p^2G''(p)+2pG'(p)=H'(p),$$
    we see from this, together with \cref{lemma:G-strictly-increasing}, that $g(p)>-2$ for all $p\geq1$. Next, we claim that $g(p)\to-2$ as $p\to\infty$. Indeed, using \cref{lemma:asymptotic-of-G}, we have that
    $$g(p)=\frac{1+O\left(\frac{1}{\log p}\right)}{-1/2+O\left(\frac{1}{\log p}\right)}\to-2\quad\text{as}\quad p\to\infty.$$
    Set now $C=\min_{p\in[1,p_0]}g(p)$. As $g(p)\to-2$ as $p\to\infty$, we can find an integer $m_0\geq2$ such that $-2<g(p)<C$ for all $p\geq m_0$. Take now any integer $m\geq m_0$ and any $p\geq 1$. If $p\in[1,p_0]$, then $g(p)\geq C>g(mp)$ by construction, and if $p\geq p_0$, then $g(p)>g(mp)$ as $g$ is decreasing on $[p_0,\infty)$. This shows \cref{eq:G-mp-p-inequality} for all $m\geq m_0$ and all $1<p<\infty$, and so $h'_m(p)<0$ for all $m\geq m_0$ and all $1<p<\infty$. The result follows.
\end{proof}

We shall in particular need the following consequence of the above lemmas.

\begin{lemma}\label{lemma:F-larger-at-given-point}
    For any relatively open and bounded subinterval $I$ of $[1,\infty)$ and any $p'\in I$, there exist positive integers $n$ and $m$ such that
    $$F_{n,m}(p')>\sup_{p\in[1,\infty]\setminus I}F_{n,m}(p).$$
\end{lemma}

\begin{proof}
    Let $m_0$ be the integer obtained by \cref{lemma:existence-of-m}. Suppose first that $I=[1,p_0)$ for some $p_0>1$ and that $p'=1$. Using \cref{lemma:G-strictly-increasing} and writing
    $$\log F_{n,m_0}(1)-\log F_{n,m_0}(p_0)=m_0(G(p_0)-G(1))\left(\frac{n}{m_0}-\frac{G(m_0p_0)-G(m_0)}{G(p_0)-G(1)}\right),$$
    we see that if we choose $n$ sufficiently large, then $F_{n,m_0}(1)>F_{n,m_0}(p_0)$; fix such an $n$. The mean value theorem and \cref{lemma:existence-of-m} then imply that $F_{n,m_0}$ is strictly decreasing on $[p_0,\infty]$, and so
    $$F_{n,m_0}(1)>F_{n,m_0}(p_0)=\sup_{p\in[p_0,\infty]}F_{n,m_0}(p).$$
    This shows the result in this case. Suppose next that $I$ is any relatively open and bounded subinterval of $[1,\infty)$ with endpoints $1\leq p_1<p_2<\infty$ and that $p'\in(p_1,p_2)$. Use \cref{lemma:G-fraction-inequality-not-at-one} to find a $t_0>0$ such that
    $$\frac{G(tp')-G(tp_1)}{G(p')-G(p_1)}-\frac{G(tp_2)-G(tp')}{G(p_2)-G(p')}>\frac{1}{t}$$
    for all $t\geq t_0$. Fix an integer $m\geq\max\{t_0,m_0\}$ and observe that the interval
    $$\left(m\frac{G(mp_2)-G(mp')}{G(p_2)-G(p')},m\frac{G(mp')-G(mp_1)}{G(p')-G(p_1)}\right)$$
    has length strictly greater than $1$, and so we can find an integer $n$ inside of it. Note that $n$ is a positive integer as a consequence of \cref{lemma:G-strictly-increasing}. Using \cref{lemma:G-strictly-increasing} we can also compute
    $$\log F_{n,m}(p')-\log F_{n,m}(p_1)=m(G(p')-G(p_1))\left(\frac{G(mp')-G(mp_1)}{G(p')-G(p_1)}-\frac{n}{m}\right)>0,$$
    and
    $$\log F_{n,m}(p')-\log F_{n,m}(p_2)=m(G(p_2)-G(p'))\left(\frac{n}{m}-\frac{G(mp_2)-G(mp')}{G(p_2)-G(p')}\right)>0,$$
    so that
    $$F_{n,m}(p')>\max\{F_{n,m}(p_1),F_{n,m}(p_2)\}.$$
    Finally, using the mean value theorem together with \cref{lemma:existence-of-m} since $m\geq m_0$, we have that $F_{n,m}$ is increasing on $[1,p_1]$ and decreasing on $[p_2,\infty]$, so that
    \[F_{n,m}(p')>\max\{F_{n,m}(p_1),F_{n,m}(p_2)\}\geq\sup_{p\in[1,\infty]\setminus I}F_{n,m}(p).\qedhere\]
\end{proof}

With these lemmas at hand, we can now perform the main part of the proof, which is the following lemma.

\begin{lemma}\label{lemma:main-lemma}
    For any relatively open and bounded subinterval $I$ of $[1,\infty)$ and any $p'\in I$, there exists a linear operator $T$ such that the following properties hold:
    \begin{enumerate}[(i)]
        \item $T$ is densely defined on $H^p(\bT^\infty)$ for all $1\leq p\leq\infty$.
        \item $T$ extends to a bounded operator on $H^p(\bT^\infty)$ for all $p\in[1,\infty]\setminus I$ with
        $$\sup_{p\in[1,\infty]\setminus I}\norm{T}_{p\to p}<\infty.$$
        \item $T$ is unbounded on $H^p(\bT^\infty)$ for all $p$ in some relatively open neighborhood of $p'$ in $I$.
        \item $T$ annihilates constants.
    \end{enumerate}
\end{lemma}

\begin{proof}
    Use \cref{lemma:F-larger-at-given-point} to find positive integers $n$ and $m$ such that
    \begin{equation}\label{eq:chosen-F-inequality}
        F_{n,m}(p')>\sup_{p\in[1,\infty]\setminus I}F_{n,m}(p)
    \end{equation}
    and consider the operator
    $$R=\frac{T_{n,m}}{\sup_{p\in[1,\infty]\setminus I}F_{n,m}(p)}.$$
    Then $R$ is a bounded operator on $H^p(\bT^\infty)$ for all $1\leq p\leq\infty$. Observe first that $p\mapsto\norm{R}_{p\to p}$ is continuous, so as $\norm{R}_{p'\to p'}>1$ by \cref{eq:chosen-F-inequality}, we can find a relatively open neighborhood $J$ of $p'$ in $I$ and a $\rho>0$ such that
    $$\norm{R}_{p\to p}\geq1+\rho$$
    for all $p\in J$. Furthermore it is clearly also the case that $\norm{R}_{p\to p}\leq1$ for all $p\in[1,\infty]\setminus I$. Next, observe that, as $\fA_d$ is self-adjoint and as $\fA_{2n}\Phi_n=\Phi_n$ $\fA_2\varphi^m=\varphi^m$, we have that $R=\fA_{2n}R\fA_{2n}$, so we may equivalently consider $R$ as an operator $R:H^p(\bT^{2n})\to H^p(\bT^{2n})$ for all $p\in[1,\infty]$. Decompose
    $$\bT^\infty=\bT^{2n}_1\times\bT^{2n}_2\times\cdots$$
    where $\bT^{2n}_k$ contains the variables $\chi_{2n(k-1)+1},\cdots\chi_{2nk}$, and define the operator $R_k$ by letting $R$ act on each of the tori $\bT^{2n}_{(k-1)k/2+1},\dots,\bT^{2n}_{k(k+1)/2}$ independently. Then
    $$\norm{R_k}_{p\to p}=\norm{R}^k_{p\to p}$$
    for each $p\in[1,\infty]$. Define the operator $T$ formally by
    $$Tf=\sum_{k=1}^\infty\frac{R_k f}{k^2}.$$
    If $p\in[1,\infty]\setminus I$ and $f\in H^p(\bT^\infty)$, then
    $$\sum_{k=1}^\infty\frac{\norm{R_k f}_p}{k^2}\leq\sum_{k=1}^\infty\frac{1}{k^2}\norm{f}_p=\frac{\pi^2}{6}\norm{f}_p,$$
    so $T$ is well-defined on $H^p(\bT^\infty)$ and defines a bounded operator of norm at most $\pi^2/6$ there. Observe next that if $f$ is a polynomial, then $R_kf=0$ for all sufficiently large $k$ as each $R_k$ acts on different variables and annihilates constants, and so $Tf$ is necessarily a polynomial. As the polynomials are dense in $H^p(\bT^\infty)$ for all $1\leq p<\infty$, we see thus that $T$ is densely defined on $H^p(\bT^\infty)$ for all $p\in[1,\infty]$. Finally, as $R_k$ annihilates constants since $f\mapsto\langle f,\Phi_n\rangle$ clearly does and as each $R_k$ acts on separate variables, we have that
    $$\norm{T}_{p\to p}\geq\frac{\norm{R_k}_{p\to p}}{k^2}\geq\frac{(1+\rho)^k}{k^2}$$
    for all $p\in J$ and all positive integers $k$. Letting $k\to\infty$ we obtain that $T$ is unbounded on $H^p(\bT^\infty)$ for all $p\in J$. Finally, that $T$ annihilates constants is clear, as each $R_k$ does.
\end{proof}

With this lemma at hand, we can now prove our main result.

\begin{proof}[Proof of \cref{theorem:main-theorem}]
    Set $B=[1,\infty)\setminus A$. Suppose first that $B$ is non-empty. As $B$ is a relatively open subset of $[1,\infty)$, we can find a sequence $\{I'_n\}_{n\in\bZ^+}$ of relatively open and bounded subintervals of $[1,\infty)$ contained in $B$ such that
    $$B=\bigcup_{n\in\bZ^+}I'_n.$$
    For any $p\in B$, find an $m(p)\in\bZ^+$ such that $p\in I'_{m(p)}$, and apply \cref{lemma:main-lemma} to find a linear operator $T_{p}$ with the properties listed in the lemma for the interval $I'_{m(p)}$ and the point $p$. In particular, for each $p\in B$, we can find a relatively open neighborhood $J_{p}$ of $p$ in $I'_{m(p)}$ such that $T_{p}$ is unbounded on $H^q(\bT^\infty)$ for all $q\in J_{p}$. Then $\{J_{p}\}_{p\in B}$ is a relatively open cover of $B$, and so, as $B$ can be written as a countable union of compact sets, we can find a countable subcover $\{J_n\}_{n\in\bZ^+}$, where we write $J_n=J_{p_n}$ for the corresponding $p_n\in B$, and similarly write $I_n=I'_{m(p_n)}$. Fix a bijection $\kappa:\bZ^+\times\bZ^+\to\bZ^+$ and decompose
    $$\bT^\infty=\bT^\infty_1\times\bT^\infty_2\times\cdots,$$
    where $\bT^\infty_n$ contains the variables $\chi_{\kappa(n,1)},\chi_{\kappa(n,2)},\dots$. Let $T_n$ be the operator obtained by letting $T_{p_n}$ act on the torus $\bT^\infty_n$ and formally define
    $$Tf=\sum_{n=1}^\infty\frac{T_nf}{n^2c_n}$$
    where $c_n=\sup_{p\in[1,\infty]\setminus I_n}\norm{T_n}_{p\to p}$. Note that if $p\in A\cup\{\infty\}$ and $f\in H^p(\bT^\infty)$, then
    $$\sum_{n=1}^\infty\frac{\norm{T_nf}_p}{n^2c_n}\leq\sum_{n=1}^\infty\frac{1}{n^2}\norm{f}_p=\frac{\pi^2}{6}\norm{f}_p,$$
    so that $T$ defines a bounded operator on $H^p(\bT^\infty)$ of norm at most $\pi^2/6$. This also implies that $T$ is densely defined on $H^p(\bT^\infty)$ for all $1\leq p\leq\infty$. Next, we claim that $T$ is unbounded on $H^p(\bT^\infty)$ for all $p\in [1,\infty)\setminus A$. Fix $p\in [1,\infty)\setminus A=B$. By construction we can then find an $n\in\bZ^+$ such that $p\in J_n$. As $T_n$ is unbounded on $H^p(\bT^\infty)$ and annihilates constants, and since each $T_k$ acts on different variables, it then follows that
    $$\norm{T}_{p\to p}\geq\frac{\norm{T_n}_{p\to p}}{n^2c_n}=\infty,$$
    and so $T$ is unbounded on $H^p(\bT^\infty)$. If $\infty\in A$, then we are done, and $T$ is the desired operator, so suppose $\infty\notin A$. By \cite{anderson_integral_2014}*{Proposition 2.14}, we can find some $g\in H^\infty(\bD)$ such that the Volterra operator $V_g$, defined by
    \begin{equation}\label{eq:volterra}
        V_gf(z)=\int_0^z f(\zeta)g'(\zeta)\,\dd\zeta,\quad z\in\bD,
    \end{equation}
    is unbounded on $H^\infty(\bD)$. It also follows from \cite{aleman_integral_1995}*{Theorem 1} that $V_g$ is bounded on $H^p(\bD)$ for all $1\leq p<\infty$. Observe also that if $P$ is a polynomial on $\bD$, then by integrating by parts in \cref{eq:volterra} shows that $V_gP\in H^\infty(\bD)$. Identifying $H^p(\bD)$ and $H^p(\bT)$, we have that
    $$T+V_g\fA_1$$
    is the desired operator. Finally, if $B$ was empty, then we may take $T=0$ as our operator if $\infty\in A$, and $T=V_g\fA_1$, with $g$ as above, if $\infty\notin A$.
\end{proof}

\section{Consequences for the finite-dimensional tori}\label{section:finite-dimensional-consequences}

Our method for proving \cref{theorem:main-theorem} also allows us to say something about the growth of the best interpolation constant for $H^p(\bT^d)$ as $d\to\infty$. For a positive integer $d$ and numbers $1\leq p_1<p<p_2\leq\infty$, we will write
\begin{multline*}
C_d(p_1,p,p_2)=\sup\{\norm{T}_{p\to p}:T\text{ is linear and bounded on }H^{p_j}(\bT^d)\\\text{ with }\norm{T}_{p_j\to p_j}\leq 1\text{ for }j=1,2\}.
\end{multline*}
From the construction used to prove \cref{theorem:main-theorem}, we can then deduce the following.

\begin{theorem}\label{theorem:asymptotic-of-finite-interpolation-constant}
    For all $1<p_1<p<p_2<\infty$ there exists an $\alpha>0$ such that
    $$C_d(p_1,p,p_2)=e^{(\alpha+o(1))d}$$
    as $d\to\infty$.
\end{theorem}

\begin{proof}
    Fix $1<p_1<p<p_2<\infty$. We claim first that
    $$\alpha=\sup_{d\in\bZ^+}\frac{1}{d}\log C_d(p_1,p,p_2)=\lim_{d\to\infty}\frac{1}{d}\log C_d(p_1,p,p_2)$$
    with the limit existing (possibly as $\infty$). By Fekete's lemma (e.g., \cite{steele_probability_1997}*{Lemma 1.2.1}), it suffices to show that $d\mapsto C_d(p_1,p,p_2)$ is supermultiplicative. Fix positive integers $d_1,d_2$ and let $T_1,T_2$ be admissible operators for $C_{d_1}(p_1,p,p_2)$ and $C_{d_2}(p_1,p,p_2)$ respectively. By letting $T$ be the operator on $H^p(\bT^{d_1+d_2})$ given by letting $T_1$ act on the first $d_1$ variables and $T_2$ act on the second $d_2$ variables, we have that $T$ is an admissible operator for $C_{d_1+d_2}(p_1,p,p_2)$. Consequently
    $$C_{d_1+d_2}(p_1,p,p_2)\geq\norm{T}_{p\to p}=\norm{T_1}_{p\to p}\norm{T_2}_{p\to p}.$$
    As $T_1$ and $T_2$ were arbitrary, the supermultiplicativity of $d\mapsto C_d(p_1,p,p_2)$ follows, so by Fekete's lemma, the claimed limit exists. We next claim that $\alpha<\infty$. Let $d$ be any positive integer and let $T$ be an admissible operator for $C_d(p_1,p,p_2)$. Applying the Riesz--Thorin interpolation theorem (e.g., \cite{bergh_interpolation_1976}*{Theorem 1.1.1}) to the operator $TP_d^+$, where $P_d^+$ is the Riesz projection on $\bT^d$, and noting that $\norm{P_d^+}_{p_j\to p_j}=\norm{P_1^+}_{p_j\to p_j}^d$, we can estimate
    $$\norm{T}_{p\to p}\leq\norm{TP_d^+}_{p\to p}\leq\left(\norm{P_1^+}_{p_1\to p_1}^{1-\theta}\norm{P_1^+}_{p_2\to p_2}^\theta\right)^d$$
    where $\theta\in(0,1)$ satisfies $1/p=(1-\theta)/p_1+\theta/p_2$. As $P_1^+$ is bounded on $L^p(\bT)$, it readily follows that $\alpha<\infty$. Finally, to see that $\alpha>0$, let $R$ be as in the proof of \cref{lemma:main-lemma} with $I=(p_1,p_2)$ and $p'=p$. Then $R$ is admissible for the constant $C_d(p_1,p,p_2)$ for some $d$, and so, as $\norm{R}_{p'\to p'}>1$, we have that
    $$\alpha\geq\frac{1}{d}\log C_d(p_1,p,p_2)\geq\frac{1}{d}\log\norm{R}_{p'\to p'}>0.$$
    The result follows.
\end{proof}

\subsection*{Acknowledgments.}
I would like to thank my PhD supervisor, Ole Fredrik Brevig, for his insightful comments and discussions during the writing of this paper.

\subsection*{AI Declaration.}
During the preparation of this manuscript, the author used OpenAI's ChatGPT 5.6 Plus as a tool for literature search, proofreading, mathematical discussion, and feedback on correctness and possible minor improvements. The main use for ChatGPT was to assist in the special function arguments used to prove the required properties of the norm of $T_{n,m}$, identifying the Volterra operator with the desired properties, and suggesting using Fekete's lemma in \cref{theorem:asymptotic-of-finite-interpolation-constant} to improve an earlier version of the result. In particular, the idea to consider the operator $T_{n,m}$ and varying $n$ and $m$ to get the desired norm inequality is the author's own. The manuscript has been written and verified by the author, and the author takes full responsibility for its content.

\bibliography{references}

\end{document}